\documentclass{amsart}
\usepackage[utf8]{inputenc}
\usepackage{geometry}
\usepackage{epsdice}
\usepackage{enumitem}
\usepackage{color}

\usepackage[OT2,T1]{fontenc}
\usepackage[final]{microtype}
\usepackage{tikz}
\usepackage{tikz-cd}
\usepackage{xypic}

\usepackage[draft=false,linktocpage=true]{hyperref}
\usepackage[capitalise]{cleveref}

\usepackage{relsize}
\usepackage[bbgreekl]{mathbbol}
\usepackage{amsfonts}
\DeclareSymbolFontAlphabet{\mathbb}{AMSb} %to ensure that the meaning of \mathbb does not change
\DeclareSymbolFontAlphabet{\mathbbl}{bbold}
\newcommand{\Prism}{{\mathlarger{\mathbbl{\Delta}}}} 

\usepackage{mathtools}
\usepackage{amsmath}

\usepackage{amsthm,amssymb}
\numberwithin{equation}{section}

\theoremstyle{plain}
\newtheorem{theorem}[equation]{Theorem}

\newtheorem{proposition}[equation]{Proposition}
\newtheorem{lemma}[equation]{Lemma}

\newtheorem{corollary}[equation]{Corollary}

\theoremstyle{definition}

\newtheorem{notation}[equation]{Notation}
\newtheorem{notations}[equation]{Notations}
\newtheorem{situation}[equation]{Situation}
\newtheorem{remark}[equation]{Remark}

\title{Bounding crystalline torsion from \'{e}tale torsion}

\author{Ofer Gabber}
\address{O.~Gabber: IH\'{E}S, 35 route de Chartres, 91440 Bures-sur-Yvette, France}
\email{gabber@ihes.fr}

\author{Shizhang Li}
\address{S.~Li:  Morningside Center of Mathematics and State Key Laboratory of Mathematical Sciences,
Academy of Mathematics and Systems Science, Chinese Academy of Sciences, Beijing 100190, China}
\email{lishizhang@amss.ac.cn}

\begin{document}

\begin{abstract}
In this note, we prove that given a smooth proper family 
over a $p$-adic ring of integers, one gets a control of its
crystalline torsion in terms of its \'{e}tale torsion,
the cohomological degree, and the ramification.
Our technical core result is a boundedness result concerning annihilator ideals of $u^{\infty}$-torsion
in Breuil--Kisin prismatic cohomology, which might be of independent interest.
\end{abstract}

\maketitle

\tableofcontents

% * Introduction
\section{Introduction}
\addtocontents{toc}{\protect\setcounter{tocdepth}{1}}
Let $p$ be a prime. 
Let $\mathcal{O}_K$ be a complete DVR of mixed characteristic
$(0, p)$ with perfect residue field $k$.
Let $\mathcal{X}$ be a smooth proper formal scheme over
$\mathrm{Spf}(\mathcal{O}_K)$.
We are interested in the interplay between two torsion phenomena associated with $\mathcal{X}$:
the \'{e}tale torsion $\mathrm{H}^i_{\acute{e}t}(\mathcal{X}_{\bar{K}}, \mathbb{Z}_p)_{\mathrm{tors}}$ and the crystalline torsion
$\mathrm{H}^i_{\mathrm{crys}}(\mathcal{X}_{k}/W(k))_{\mathrm{tors}}$.
In Bhatt--Morrow--Scholze's first paper \cite{BMS1},
one learns the following:
\begin{theorem}[{\cite[Theorem 1.1.(ii)]{BMS1}}]
There is an inequality
\[
\mathrm{length}\left(\mathrm{H}^i_{\acute{e}t}(\mathcal{X}_{\bar{K}}, \mathbb{Z}_p)_{\mathrm{tors}}\right)
\leq \mathrm{length}\left(\mathrm{H}^i_{\mathrm{crys}}(\mathcal{X}_{k}/W(k))_{\mathrm{tors}}\right).
\]
In particular, if $\mathrm{H}^i_{\mathrm{crys}}(\mathcal{X}_{k}/W(k))_{\mathrm{tors}} = 0$, then $\mathrm{H}^i_{\acute{e}t}(\mathcal{X}_{\bar{K}}, \mathbb{Z}_p)_{\mathrm{tors}} = 0$.
\end{theorem}
It is natural to wonder about the converse question:
if $\mathrm{H}^i_{\acute{e}t}(\mathcal{X}_{\bar{K}}, \mathbb{Z}_p)_{\mathrm{tors}} = 0$, then what can we say about
$\mathrm{H}^i_{\mathrm{crys}}(\mathcal{X}_{k}/W(k))_{\mathrm{tors}}$?
The K\"{u}nneth formula and examples in \cite[Section 2]{BMS1}
shows that one cannot get any bound on the length
of the $\mathrm{H}^i_{\mathrm{crys}}(\mathcal{X}_{k}/W(k))_{\mathrm{tors}}$.
In this paper we show that one can get a bound of the exponent
of the $\mathrm{H}^i_{\mathrm{crys}}(\mathcal{X}_{k}/W(k))_{\mathrm{tors}}$,
defined as the smallest natural number $m$ such that $p^m \cdot \mathrm{H}^i_{\mathrm{crys}}(\mathcal{X}_{k}/W(k))_{\mathrm{tors}} = 0$.

\begin{theorem}[= \Cref{bounding crystalline torsion}]
There is a constant $c(e, i)$ depending only on the ramification index
$e = v_K(p)$ and the cohomological degree $i > 0$, 
such that there is an inequality
\[
\mathrm{exp}(\mathrm{H}^i_{\mathrm{crys}}(\mathcal{X}_k/W)_{\mathrm{tors}}) \leq
\mathrm{exp}(\mathrm{H}^i_{\acute{e}t}(\mathcal{X}_C, \mathbb{Z}_p)_{\mathrm{tors}}) + c(e, i).
\]
\end{theorem}

Our technical tool is a generalization of some results in
\cite{LL23}, concerning the annihilator ideal of $u^\infty$-torsion
in prismatic cohomology of $\mathcal{X}$.
This is the content of our \Cref{CA section}.
In \Cref{AG section}, we give some applications of the bound of
aforementioned annihilator ideals, and end with a proof of the above theorem.

\subsection*{Notations and Conventions}
Let $k$ be a perfect field of characteristic $p$, let $W = W(k)$ be its ($p$-typical) Witt ring.
Denote $\mathfrak{S} \coloneqq W [\![u]\!]$ equipped with $(u, p)$-adically continuous Frobenius
$\varphi \colon \mathfrak{S} \to \mathfrak{S}$ such that $\varphi|_W$ is the usual Witt vector Frobenius
and $\varphi(u) = u^p$. 
Lastly let $E(u) \in \mathfrak{S}$ be an Eisenstein polynomial of degree $e$.

\newpage

\section{Some commutative algebra arguments}
\label{CA section}
Throughout this section, we shall consider the following situation.
\begin{situation}
\label{annihilator situation}
Let $J \subset \mathfrak{S}$ be an ideal and let $j \in \mathbb{N}$, satisfying
\begin{enumerate}
\item the ideal $J$ is cofinite, namely $(p,u)^N \subset J$ for some $N$; and
\item we have a containment relation $E^j \cdot J \subset \varphi(J) \cdot \mathfrak{S}$.
\end{enumerate}
In this situation, let us denote
$J + (p) = (p, u^{\sigma})$ and $J + (u) = (u, p^{\rho})$.
It is easy to see that $\sigma \leq \lfloor\frac{e \cdot j}{p-1}\rfloor$,
see for instance the proof of \cite[Corollary 3.4]{LL23}.
\end{situation}
The aim of this section is to give explicit estimate of $\rho$ in terms of $e$ and $j$.

\subsection{Argument one}
In this subsection, we present the first argument.
\begin{notation}
Let $c(a, b) \coloneqq \min\{c \in \mathbb{N} \mid p^c \in (u^a, E^b)\}$.
\end{notation}

\begin{lemma}
\label{estimate of c}
We have that $c(a, b) \leq \lceil\frac{a}{e}\rceil + b - 1$.
\end{lemma}

\begin{proof}
By assumption $E = u^e + p \cdot \text{unit}$, so $p \in (u^e, E)$.
More generally we have $p^{x + y - 1} \in (u^{ex}, E^y)$.
\end{proof}

\begin{lemma}
\label{argument one lemma}
Let $J \subset \mathfrak{S}$ and $j \in \mathbb{N}$ be as in
\Cref{annihilator situation}.
Suppose that $J + (u^a) \subset (u^a, p^N)$, then 
%$J + (u^{A}) \subset (u^{A}, p^{\max\{0, N - c(A, j)\}})$ for any natural number $A \leq pa$.
$J \subset (u^{A}, p^{\max\big(0, N - c(A, j)\big)})$ for any natural number $A \leq pa$.
\end{lemma}

\begin{proof}
In the ring $R \coloneqq \mathfrak{S}/u^{A}$, we have
\[
p^{c(A, j)} \cdot J R \subset E^j \cdot J R \subset \varphi(J) \cdot R \subset (p^N).
\]
Our claim follows from the fact that the sequence $(u, p)$ is $\mathfrak{S}$-regular.
\end{proof}

\begin{proposition}
\label{Conclusion in argument one}
Let $J \subset \mathfrak{S}$ and $j \in \mathbb{N}$ be as in
\Cref{annihilator situation}. Let $a_1, a_2, \ldots, a_n$ be a sequence of integers satisfying
\begin{enumerate}
\item $a_0 = 1$;
\item $a_i \leq p \cdot a_{i-1}$;
\item $a_n > \frac{e \cdot j}{p-1}$.
\end{enumerate}
Then $\rho \leq \sum_{i = 1}^{n} c(a_i, j)$.

In particular, if $e \cdot j < p^n(p-1)$, then we may choose
$a_i = p^i$ for $i \leq (n-1)$ and $a_n = \lfloor \frac{e \cdot j}{p-1} \rfloor + 1$,
hence $\rho \leq \sum_{i = 1}^{n-1} \lfloor \frac{p^i}{e} \rfloor + 
\lfloor \frac{\lfloor \frac{e \cdot j}{p-1} \rfloor + 1}{e}\rfloor + nj$.
\end{proposition}

\begin{proof}
The second sentence follows from the first one and \Cref{estimate of c}
as $\lceil x \rceil - 1 < x$.
To see the first sentence: Let $J + (u) = (u, p^{\rho})$, and 
assume to the contrary that $\rho > \sum_{i = 1}^{n} c(a_i, j)$.
Then applying \Cref{argument one lemma}, we see that
$J + (u^{a_1}) \subset (u^{a_1}, p^{\rho - c(a_1, j)})$.
Applying \Cref{argument one lemma} again, we see that
$J + (u^{a_2}) \subset (u^{a_2}, p^{\rho - c(a_1, j) - c(a_2, j)})$.
Repeating the above, we finally see that
$J + (u^{a_n}) \subset (u^{a_n}, p^{> 0})$. But this contradicts to the fact that
$J + (p) = (p, u^{\sigma})$ with $\sigma \leq \frac{e \cdot j}{p-1} < a_n$.
\end{proof}

\subsection{Argument two}
In this subsection, we present the second argument.
Throughout the subsection, let $J \subset \mathfrak{S}$ and $j \in \mathbb{N}$ be as in
\Cref{annihilator situation}.

\begin{lemma}
Let $r \in [0, \infty)$ be a real number, the following map
\[
v_r \colon \mathfrak{S}\setminus\{0\} \to \mathbb{R},~v_r(\sum_i a_i u^i) \coloneqq \min\{\mathrm{ord}_p(a_i) + i \cdot r\}
\]
defines an additive valuation.
\end{lemma}

\begin{proof}
It is easy to check that minimum is always attained, one can check the triangle inequality
\[
v_r\left((\sum_i a_i u^i) + (\sum_i b_i  u^i)\right) \geq \min\bigg(v_r(\sum_i a_i u^i),v_r(\sum_i b_i u^i)\bigg)
\]
using the definition.
Lastly we need to check multiplicativity:
\[
v_r\left((\sum_i a_i u^i)\cdot(\sum_i b_i  u^i)\right) = v_r(\sum_i a_i u^i) + v_r(\sum_i b_i u^i).
\]
One checks directly that the multiplicativity holds true if one of the power series is just a monomial.
Now let $\alpha \coloneqq \min\{i \in \mathbb{N} \mid \mathrm{ord}_p(a_i) + i \cdot r = v_r(\sum_i a_i u^i)\}$
and $\beta \coloneqq \min\{i \in \mathbb{N} \mid \mathrm{ord}_p(b_i) + i \cdot r = v_r(\sum_i b_i u^i)\}$.
Using the definition, one checks that
\[
v_r\left((\sum_{i \geq \alpha} a_i u^i)\cdot(\sum_{i\geq \beta} b_i  u^i)\right) = 
v_r(\sum_i a_i u^i) + v_r(\sum_i b_i u^i).
\]

Finally, by combining
\begin{itemize}
\item the case of one of the power series being a monomial;
\item the decompositions
$\sum_i a_i u^i = \sum_{i<\alpha} a_iu^i + \sum_{i\geq\alpha} a_iu^i$
and $\sum_i b_i u^i = \sum_{i<\beta} b_iu^i + \sum_{i\geq\beta} b_iu^i$ of the two power series;
\item the above equality; and
\item the triangle inequality,
\end{itemize}
one arrives at the multiplicativity statement which finishes the proof.
\end{proof}

One may view the ring $\mathfrak{S}$ as the analytic functions bounded by $1$ on the open unit disc
$\mathbb{D}^{\circ}_{W[1/p]}$, then the valuation $v_r$ corresponds to the 
Gauss norm on the radius $p^{-r}$ disc (the absolute value is normalized so that $|p| = p^{-1}$).
Notice that for $r > 0$, we can take a rational number $s \in (0, r]$, 
so the said Gauss norm is a rank $1$ point on the closed disc of radius $p^{-s}$
around $0$.
Therefore, we may view it as a rank $1$ point on the open unit disc, giving rise to
a norm on $\mathfrak{S}[1/p]$ whose restriction to $\mathfrak{S}$ is bounded by $1$.

\begin{notations}
For any co-finite ideal $I \subset \mathfrak{S}$,
let $f_I(r) \coloneqq v_r(I)$, viewed as a function $f_I\colon [0,\infty) \to \mathbb{R}_{\geq 0}$.
Let $I^{\mathrm{mon}} \coloneqq \text{ the ideal generated by }
\{a_i u^i \mid \sum_i a_iu^i \in I\}$.
\end{notations}

Namely for every power series in $I$, we extract out all of its monomial terms, then we use all
these monomial terms of all elements in $I$ to generate a (most likely larger) ideal.
Note that $I^{\mathrm{mon}}$ is generated by finitely many monomial terms as $\mathfrak{S}$ is Noetherian.

\begin{lemma}
\label{property of characteristic function}
Let $I \subset \mathfrak{S}$ be a co-finite ideal, we have natural numbers $\sigma$ and $\rho$
satisfying $I + (p) = (p, u^{\sigma})$ and $I + (u) = (u, p^{\rho})$.
Then the function $f_I$ satisfies the following:
\begin{enumerate}
\item We have an equality $f_I = f_{I^{\mathrm{mon}}}$;
\item The function $f_I$ is concave and continuous;
\item The function $f_I$ is piecewise linear, on each interval it is given by $a\cdot r + b$ with both
$a$ and $b$ natural numbers;
\item There exists an $\epsilon > 0$, such that
\[
f_I(r)= \begin{cases} 
      \sigma \cdot r, & r \in [0, \epsilon], \\
      \rho, & r\in [1/\epsilon,\infty).
\end{cases}
\]
\item We have an equality $f_{\varphi(I)}(r) = f_I(p \cdot r)$.
\end{enumerate}
\end{lemma}

\begin{proof}
(1) and (5) follows from the definition of $v_r$.
Our assumption implies that 
\[
I^{\mathrm{mon}} = 
(p^{\rho}, a_1 \cdot u, a_2 \cdot u^2, \ldots, a_{\sigma -1}u^{\sigma -1}, u^{\sigma}),
\]
where $\mathrm{ord}_p(a_i) > 0$ (and $a_i$ is allowed to be $0$).
For each of the generators above, if we look at their $v_r$ as a function in $r$,
we simply get a linear function with a natural number slope.
The function $f_I = f_{I^{\mathrm{mon}}}$ is minimum of the above collection of linear functions,
this immediately gives us (2) and (3).
Using (1) and the definition of $v_r$, we also see that 
$v_r(I) = v_r(u^{\sigma})$ if $r$ is sufficiently near $0$
and $v_r(I) = v_r(p^{\rho})$ if $r \gg 0$, which proves (4).
\end{proof}

\begin{lemma}
\label{property taking Frob into account}
Let $J \subset \mathfrak{S}$ and $j \in \mathbb{N}$ be as in
\Cref{annihilator situation}. We have
\begin{enumerate}
\item the function $g(r) \coloneqq v_r(E^j) = \min\bigg((e \cdot j) \cdot r, j\bigg)$; and
\item an inequality $f_J(p \cdot r) \leq f_J(r) + g(r)$.
\end{enumerate}
\end{lemma}

\begin{proof}
(1) easily follows from our assumption on the degree $e$ Eisenstein polynomial $E$.
(2) follows from the assumption $E^j \cdot J \subset \varphi(J) \cdot \mathfrak{S}$
and \Cref{property of characteristic function}.(5).
\end{proof}

\begin{lemma}
\label{comparison function}
Let $J \subset \mathfrak{S}$ and $j \in \mathbb{N}$ be as in
\Cref{annihilator situation}. Define a piecewise linear function
\[
h(r) = \begin{cases}
    \sigma \cdot r, & r \in [0, \frac{p \cdot j}{\sigma \cdot (p-1)}] \\
    \frac{\sigma}{p} \cdot r + j, & r \in [\frac{p \cdot j}{\sigma \cdot (p-1)}, \frac{p^2 \cdot j}{\sigma \cdot (p-1)}] \\
    \frac{\sigma}{p^2} \cdot r + 2 \cdot j, & r \in [\frac{p^2 \cdot j}{\sigma \cdot (p-1)}, \frac{p^3 \cdot j}{\sigma \cdot (p-1)}] \\
    \ldots & \\
    \frac{\sigma}{p^n} \cdot r + n \cdot j, & r \in [\frac{p^n \cdot j}{\sigma \cdot (p-1)}, \frac{p^{n+1} \cdot j}{\sigma \cdot (p-1)}] \\
    \ldots &
\end{cases}.
\]
Then we have $f_J(r) \leq h(r)$.
\end{lemma}

We leave it to the readers to check that the $h(r)$ above is continuous, concave and increasing.

\begin{proof}
Let us check inductively on each interval that $f_J(r) \leq h(r)$. For the first interval $[0, \frac{p \cdot j}{\sigma \cdot (p-1)}]$,
we need to show $f_J(r) \leq \sigma \cdot r$, this follows from \Cref{property of characteristic function}.(2)-(4).
Now we prove the induction step, so we assume that $f_J(x) \leq h(x)$ whenever $x \in [0, \frac{p^{n} \cdot j}{\sigma \cdot (p-1)}]$
and let $r \in [\frac{p^{n} \cdot j}{\sigma \cdot (p-1)}, \frac{p^{n+1} \cdot j}{\sigma \cdot (p-1)}]$.
Our assumption on $r$ implies that $f_J(\frac{r}{p}) \leq h(\frac{r}{p}) = \frac{\sigma}{p^{n-1}} \cdot \frac{r}{p} + (n-1) \cdot j$.
By \Cref{property taking Frob into account}, we see that 
\[
f_J(r) \leq f_J(\frac{r}{p}) + j \leq \frac{\sigma}{p^{n-1}} \cdot \frac{r}{p} + (n-1) \cdot j + j = \frac{\sigma}{p^n} \cdot r + n \cdot j = h(r).
\]
\end{proof}

\begin{lemma}
\label{characteristic function constant here}
Let $J \subset \mathfrak{S}$ and $j \in \mathbb{N}$ be as in
\Cref{annihilator situation}. Then $f_J(r) = \rho$ whenever
$r \geq \frac{p \cdot j}{p-1}$.
\end{lemma}

\begin{proof}
Let us denote by $f'_J(r)$ the left derivative of $f_J(r)$, this is 
a piecewise constant, decreasing, eventually $0$ function, which
takes values in natural numbers, thanks to \Cref{property of characteristic function}.(2)-(4).
Therefore all we need to show is that $f'_J(r) = 0$ for $r > \frac{p \cdot j}{p-1}$.
Now \Cref{property taking Frob into account} implies that $f'_J(r) \cdot (r - \frac{r}{p}) \leq f_J(r) - f_J(\frac{r}{p}) \leq j$.
Therefore $f'_J(r) < 1$ and is a natural number, hence must be $0$.
\end{proof}

\begin{proposition}
\label{Conclusion in argument two}
Let $J \subset \mathfrak{S}$ and $j \in \mathbb{N}$ be as in
\Cref{annihilator situation}. If $e \cdot j \leq p^n(p-1)$, then we have
$\rho \leq (\frac{\sigma}{p^{n-1}(p-1)} + n) \cdot j \leq (\frac{\lfloor \frac{e \cdot j}{p-1} \rfloor}{p^{n-1}(p-1)} + n) \cdot j$.
\end{proposition}

\begin{proof}
By \Cref{characteristic function constant here}, we have $\rho = f_J(\frac{p \cdot j}{p-1})$.
Since $\sigma \leq \lfloor \frac{e \cdot j}{p-1} \rfloor \leq p^n$, we see that $\frac{p \cdot j}{p-1} \leq \frac{p^{n+1} \cdot j}{\sigma \cdot (p-1)}$
(and we only need to prove the first inequality).
Now by \Cref{comparison function}, we have
\[
\rho = f_J(\frac{p \cdot j}{p-1}) \leq h(\frac{p \cdot j}{p-1}) \leq \frac{\sigma}{p^n} \cdot \frac{p \cdot j}{p-1} + n \cdot j =
(\frac{\sigma}{p^{n-1}(p-1)} + n) \cdot j.
\]
\end{proof}

\subsection{Conclusions}

Let us first extract a concrete estimate of $\rho$ in a special case.

\begin{proposition}
\label{special case of j=1}
Let $J \subset \mathfrak{S}$ be as in
\Cref{annihilator situation}, with $j = 1$, and let $n \in \mathbb{N}$.
\begin{enumerate}
\item If $p \not= 2$ and $e < p^{n}(p-1)$, then $\rho \leq n$.
\item If $p = 2$ and $e < 2^{n}$, then $\rho \leq (n+1)$.
\end{enumerate}
\end{proposition}

Note that when $e \leq (p-1)$, our statement follows from the proof of 
\cite[Proposition 3.5]{LL23}.
So in the proof below, we always assume further that $e > (p-1)$,
in particular $n \geq 1$.

\begin{proof}
First let us assume that $p \not = 2$. Suppose that $e < p^{n-1}(p-1)^2$, then by \Cref{Conclusion in argument two}, we see that the integer
$\rho < n+1$, therefore $\rho \leq n$.
If $p^{n-1}(p-1)^2 \leq e < p^n(p-1)$, then
\begin{itemize}
\item we have $\lfloor \frac{p^i}{e} \rfloor = 0$
for all $0 \leq i \leq (n-1)$ as $p \not= 2$;
\item similarly $\lfloor \frac{\lfloor \frac{e}{p-1} \rfloor + 1}{e}\rfloor \leq \lfloor \frac{1}{p-1} + \frac{1}{e}\rfloor = 0$,
as we have assumed that $e > (p-1)$.
\end{itemize}
Therefore by \Cref{Conclusion in argument one}, we have that $\rho \leq n$.

Now in case $p = 2$, the relevant formulas simplify.
When $2^{n-1} < e < 2^n$, we get $\rho \leq (n+1)$ by \Cref{Conclusion in argument one}.
When $e = 2^{n-1}$, we get $\rho \leq (n+1)$ by \Cref{Conclusion in argument two}.
\end{proof}

Let us summarize the outcome of the previous two subsections.
\begin{notation}
For each pair of positive integers $(e, j)$, we denote
\[
d(e, j) \coloneqq \min\bigg(\sum_{i = 1}^{n-1} \lfloor \frac{p^i}{e} \rfloor + 
\lfloor \frac{\lfloor \frac{e \cdot j}{p-1} \rfloor + 1}{e}\rfloor + nj, (\frac{\lfloor \frac{e \cdot j}{p-1} \rfloor}{p^{n-1}(p-1)} + n) \cdot j\bigg),
\]
where $n$ is the smallest natural number such that
$e \cdot j < p^n(p-1)$.
\end{notation}

\begin{proposition}
\label{bound on rho}
Let $J \subset \mathfrak{S}$ and $j \in \mathbb{N}$ be as in
\Cref{annihilator situation}. Then we have $\rho \leq d(e, j)$.
\end{proposition}

\begin{proof}
Combine \Cref{Conclusion in argument one}
and \Cref{Conclusion in argument two}.
\end{proof}

\subsection{Argument for boundedness}
Lastly let us show that an additional condition gives rise to boundedness of length of
$\mathfrak{S}/J$.

\begin{proposition}
\label{boundedness proposition}
Let $J \subset \mathfrak{S}$ and $j \in \mathbb{N}$ be as in
\Cref{annihilator situation}.
Assume moreover that there is an $\ell \in \mathbb{N}$ such that
$E^\ell \cdot \varphi(J) \subset J$, then $p^{(\rho + \max(j, \ell)) \cdot \sigma} \in J$.
The additional assumption implies that
$\mathrm{length}(\mathfrak{S}/J) \leq (\rho + \max(j, \ell)) \cdot \sigma^2$,
in particular $(u, p)^{(\rho + \max(j, \ell)) \cdot \sigma^2} \subset J$.
\end{proposition}

\begin{proof}
For any ideal $I \subset \mathfrak{S}$, we denote $(I : p) \coloneqq \{f \in \mathfrak{S} \mid p \cdot f \in I\}$.
Alternatively, the ideal is defined via the following exact sequence:
\[
0 \to (I : p) \to \mathfrak{S} \xrightarrow{\cdot p} \mathfrak{S}/I.
\]
Since $(E, p)$ is a regular sequence, one checks that $E \cdot (I : p) = (E \cdot I : p)$.
Using the fact that $\varphi$ is flat, one checks that $\varphi(I : p) = (\varphi(I) : p)$.
Therefore if we let $J_0 = J$ and inductively define $J_i = (J_{i-1} : p)$ for all $i \geq 1$,
then we can make the following observations:
\begin{enumerate}
\item We have $\mathfrak{S}/J_n \xrightarrow[\cong]{\cdot p^n} p^n \cdot \mathfrak{S}/J$,
hence $\mathfrak{S}/(J_n + (p)) \xrightarrow[\cong]{\cdot p^n}
\frac{p^n \cdot \mathfrak{S}/J}{p^{n+1} \cdot \mathfrak{S}/J}$;
\item The ideals $J_n$ again satisfy conditions: $E^j \cdot (-) \subset \varphi(-) \cdot \mathfrak{S}$
and $E^{\ell} \cdot \varphi(-) \subset (-)$.
\end{enumerate}
Our task is to show that $J_n = \mathfrak{S}$ when $n \geq (\rho + \max(j, \ell)) \cdot \sigma$.
Letting $\sigma_n$ and $\rho_n$ be defined by $J_n + (p) = (p,u^{\sigma_n})$
and $J_n + (u) = (u, p^{\rho_n})$, it suffices to show that
$\sigma_i - \sigma_{i + \rho + \max(j, \ell)} \geq 1$.
Since $\rho_n$ is non-increasing, using the observation (2) above, 
it suffices to prove the above with $i = 0$.

Suppose to the contrary we have
$0 <\sigma_0 = \ldots = \sigma_{\rho + \max(j, \ell)}$, we need to deduce a contradiction.
This assumption, together with the observation (1) above, implies that multiplication
by $p$ map on $A \coloneqq \mathfrak{S}/J$ induces isomorphisms:
\[
A/pA \xrightarrow[\cong]{\cdot p} pA/p^2A \xrightarrow[\cong]{\cdot p} \ldots
\xrightarrow[\cong]{\cdot p} p^{\rho + \max(j, \ell)}A/p^{\rho + \max(j, \ell) + 1}A.
\]
Weierstrass preparation and the definition of $\sigma$ implies the existence of a polynomial
$f \in J$ such that $f \equiv u^{\sigma} \mod{p}$.
Since $(f, p)$ is a regular sequence, one checks that the $p$-adic filtration on
$B \coloneqq \mathfrak{S}/f$ also satisfies
$B/pB \xrightarrow[\cong]{\cdot p} pB/p^2B \xrightarrow[\cong]{\cdot p} \ldots$.
Let us now look at the map
$\mathfrak{S}/(f, p^{\rho + \max(j, \ell) + 1})
\twoheadrightarrow \mathfrak{S}/(J, p^{\rho + \max(j, \ell) + 1})$,
it is an isomorphism modulo $p$ so, by the above knowledge of $p$-adic filtrations
on both sides, it is an isomorphism.
Therefore we have $J \equiv (f) \mod{p^{\rho + \max(j, \ell) + 1}}$.
Moreover the definition of $\rho$ implies that the constant term of $f$
must have $p$-adic valuation $\rho$.
Now our conditions imply that there exists polynomials $P(u), Q(u) \in W/p^{\rho + \max(j, \ell) + 1}[u]$
such that we have equalities
\[
E(u)^j \cdot f = \varphi(f) \cdot P(u) \text{ and }
E(u)^\ell \cdot \varphi(f) = f \cdot Q(u)
\]
in $W/p^{\rho + \max(j, \ell) + 1}[u]$.
Now the constant term of left hand side of both equations are nonzero in $W/p^{\rho + \max(j, \ell) + 1}$,
therefore the Newton polygon of $E(u)^j \cdot f$ is the same as that of
$\varphi(f) \cdot \widetilde{P}(u)$ where $\widetilde{P}(u) \in W[u]$ is an arbitrary lift of $P(u)$.
Consequently we see that there is an inclusion of sets:
\[
\{p\text{-adic valuations of roots of } \varphi(f)\} \subset 
\{p\text{-adic valuations of roots of } f\} \cup \{1/e\}.
\]
Similarly we also have an inclusion of sets:
\[
\{p\text{-adic valuations of roots of } f\} \subset 
\{p\text{-adic valuations of roots of } \varphi(f)\} \cup \{1/e\}.
\]
Since we have an equality of subsets of $\mathbb{Q}$:
\[
1/p \cdot \{p\text{-adic valuations of roots of } f\} = \{p\text{-adic valuations of roots of } \varphi(f)\},
\]
we arrive at the following contradiction:
\[
\{p\text{-adic valuations of roots of } f\} \cup \{1/e\} = 
(1/p \cdot \{p\text{-adic valuations of roots of } f\}) \cup \{1/e\}.
\]
Therefore we see that we cannot have $(\rho + \max(j, \ell))$ many $\sigma$'s
being all equal, which finishes the proof.
\end{proof}

\section{Some prismatic cohomology facts}
In this section, we recall some statements concerning torsion in prismatic cohomology.
Let $\mathcal{X}$ be a smooth proper formal scheme over $\mathrm{Spf}(\mathcal{O}_K)$.

\begin{remark}
\label{remark on structure of BK module}
Recall (see \cite[Proposition 4.3]{BMS1} and \cite[Theorem 1.8.6]{prism}) 
that the prismatic cohomology $\mathfrak{M}^i \coloneqq \mathrm{H}^i_{\Prism}(\mathcal{X}/\mathfrak{S})$,
being a Breuil--Kisin module, admits the following canonical exact sequences:
\[
0 \to \mathfrak{M}^i_{\mathrm{tors}} = \mathfrak{M}^i[p^\infty] \to \mathfrak{M}^i \to \mathfrak{M}^i_{\mathrm{tf}} \to 0,
\]
\[
0 \to \mathfrak{M}^i_{\mathrm{tf}} \to (\mathfrak{M}^i)^{\vee \vee} \to \overline{\mathfrak{M}^i} \to 0,
\]
where $(\mathfrak{M}^i)^{\vee \vee}$ is the double dual (or reflexive hull) of $\mathfrak{M}^i$ which is finite free over $\mathfrak{S}$
and $\overline{\mathfrak{M}^i}$ is supported at the closed point $(p, u)$ of $\mathrm{Spec}(\mathfrak{S})$.
\end{remark}

The following result is the main reason why we studied the kind
of ideal $J$ in \Cref{annihilator situation}.

\begin{proposition}
\label{Consequence of Nygaard filtrations}
Let $n \in \mathbb{N} \cup \{\infty\}$,
denote $\mathfrak{M}^i_n \coloneqq \mathrm{H}^i(\mathrm{R\Gamma}_{\Prism}(\mathcal{X}/\mathfrak{S})/^L p^n))$ 
(where $n = \infty$ means that we do not perform the reduction at all).
Then we have the following:
\begin{enumerate}
\item For all $i \geq 0$, there exists maps
$F \colon \varphi_{\mathfrak{S}}^*\mathfrak{M}^i_n \to \mathfrak{M}^i_n$
and $V \colon \mathfrak{M}^i_n \to \varphi_{\mathfrak{S}}^*\mathfrak{M}^i_n$
such that both $F \circ V$ and $V \circ F$ are the same as multiplication by $E^i$;
\item For all $i > 0$, multiplication by $E^{i-1}$ on
$\varphi_{\mathfrak{S}}^*\mathfrak{M}^i_n$ factors through a submodule of
$\mathfrak{M}^i_n$.
\end{enumerate}
In particular, when $i > 0$,
let $J$ be the annihilator ideal of $\mathfrak{M}^i_n[u^\infty]$. 
Then the ideal $J$ and $(j, \ell) = (i-1, i)$ satisfy the conditions
in \Cref{annihilator situation} and \Cref{boundedness proposition}.
\end{proposition}

When $n = \infty$, the statement (1) follows from \cite[Theorem 1.8.(6)]{prism}.
In general, both (1) and (2) follow from the observation made in \cite[Proposition 3.2]{LL23}.
For the convenience of the readers, let us sketch the argument below.

\begin{proof}
Recall that the Frobenius-twisted prismatic cohomology
admits Nygaard filtrations, see \cite[Section 15]{prism}.
In particular, for any $j \geq 0$, there are natural maps
$\mathrm{R\Gamma}(\mathcal{X}_{\mathrm{qsyn}}, \mathrm{Fil_N}^j/p^n)
\to \varphi_{\mathfrak{S}}^* \mathrm{R\Gamma}(\mathcal{X}_{\mathrm{qsyn}}, \Prism/p^n)$
and
$\varphi_{\mathfrak{S}}^* \mathrm{R\Gamma}(\mathcal{X}_{\mathrm{qsyn}}, \Prism/p^n) \to 
\mathrm{R\Gamma}(\mathcal{X}_{\mathrm{qsyn}}, \mathrm{Fil_N}^j/p^n)$
such that compositions either way are the same as multiplication by $E^j$.
Moreover these Nygaard filtrations admit divided Frobenius maps to prismatic cohomology:
$\mathrm{R\Gamma}(\mathcal{X}_{\mathrm{qsyn}}, \mathrm{Fil_N}^j/p^n) \xrightarrow{\varphi_j}
\mathrm{R\Gamma}(\mathcal{X}_{\mathrm{qsyn}}, \Prism/p^n)$.

By \cite[Lemma 7.8.(3)]{LL25}, the induced map
$\mathrm{H}^j(\mathcal{X}_{\mathrm{qsyn}}, \mathrm{Fil_N}^j/p^n) \xrightarrow{\varphi_j}
\mathrm{H}^j(\mathcal{X}_{\mathrm{qsyn}}, \Prism/p^n)$
is an isomorphism.
This gives (1) by considering $i$-th Nygaard filtration.
Also by \cite[Lemma 7.8.(3)]{LL25}, when $j > 0$, the induced map
$\mathrm{H}^j(\mathcal{X}_{\mathrm{qsyn}}, \mathrm{Fil_N}^{j-1}/p^n) \xrightarrow{\varphi_{j-1}}
\mathrm{H}^j(\mathcal{X}_{\mathrm{qsyn}}, \Prism/p^n)$
is injective.
This gives (2) by considering $(i-1)$-st Nygaard filtration.
The last sentence is a consequence of (1) and (2).
\end{proof}

\begin{remark}
Let us take the opportunity to correct an error in \cite[Lemma 7.8.(3)]{LL25}.
The proof has a gap in its last sentence: namely, when we use the same proof strategy
to run the argument for proving the derived mod $p^m$ versions, the cohomological
estimate might be off by $1$ cohomological degree due to $p$-torsion
in $\Omega^{i+1}_{X/(A/I)}$, and this $p$-torsion subsheaf is nonzero exactly
when $A/I$ contains $p$-torsion (and $X/(A/I)$ has relative dimension at least $i+1$).
Therefore, by the proof strategy of loc.~cit.~we get the following conclusion:
The statement of \cite[Lemma 7.8.(1)-(3)]{LL25} is correct as is, 
but for their derived mod $p^m$ analogs,
one needs an extra assumption that $(A, I)$ is a transversal prism
(namely $A/I$ is $p$-torsion free).
So, one just needs to change the last sentence to ``Moreover their derived mod $p^m$
counterparts hold \emph{as long as} $(A, I)$ \emph{is transversal}.
Fortunately, the Breuil--Kisin prism is an example of such, which justifies our usage
of \cite[Lemma 7.8.(3)]{LL25} in the above proof.
Lastly we point out that in the proof of \cite[Lemma 7.8.(3)]{LL25},
the authors give a reference to \cite[Theorem 15.2.(2)]{prism} for the cohomological estimate,
but the more appropriate reference seems to be rather \cite[Theorem 15.3]{prism}.
\end{remark}

The rest of this section concerns the $A_{\inf}$ cohomology defined in \cite[Theorem 1.8]{BMS1},
let us recall some key definitions and properties below.

\begin{notations}
\label{tilt notation}
Let $C$ be the completion of an algebraic closure of $K$,
with its tilt $C^{\flat}$ defined as follows:
Consider the ring of integers $\mathcal{O}_C \subset C$, then define
$\mathcal{O}_C^{\flat} \coloneqq \lim_{\varphi} (\mathcal{O}_C/p)$.
Given a sequence of elements $\{x_i\}_{i \in \mathbb{N}}$ of
$\mathcal{O}_C/p$ satisfying $x_i^p = x_{i-1}$, 
we denote by $\underline{x}$ its corresponding element in $\mathcal{O}_C^{\flat}$.
It is a fact that $\mathcal{O}_C^{\flat}$ is a rank $1$ valuation ring, whose fraction field 
$\mathrm{Frac}(\mathcal{O}_C^{\flat}) \eqqcolon C^{\flat}$ is an algebraically closed
complete non-archimedean field of equal characteristic $p$.
The maximal ideal of $\mathcal{O}_C^{\flat}$ is given by 
$\mathfrak{m}_C^{\flat} = \{\underline{x} \in \mathcal{O}_C^{\flat} \mid x_0 \in \mathfrak{m}_C/(p \cdot \mathcal{O}_C) \subset \mathcal{O}_C/p\}$.
For more on this, we refer readers to \cite[Section 3]{Sch12}.

Fix a choice of compatible $p$-power primitive roots of unity $(1, \zeta_p, \zeta_{p^2}, \cdots)$,
then the sequence $\{\zeta_{p^i}\}_{i \in \mathbb{N}}$ defines
an element $\epsilon \in \mathcal{O}_C^{\flat}$.
The Fontaine period ring $A_{\inf}$ is defined as the ($p$-typical)
Witt ring of $\mathcal{O}_C^{\flat}$, equipped with Frobenius automorphism $\varphi$.
The following two elements $\mu \coloneqq [\epsilon] - 1$ and 
$\widetilde{\xi} =\varphi(\xi) = \frac{\varphi(\mu)}{\mu}$
in $A_{\inf}$ are important to us.
\end{notations}

In the rest of this section $C$ can be any algebraically closed complete 
non-archimedean field of mixed characteristic $(0,p)$.

\begin{remark}
\label{Ainf remark}
Let $\mathfrak{X}$ be a smooth proper formal scheme over $\mathrm{Spf}(\mathcal{O}_C)$
with its rigid generic fiber $X \coloneqq \mathfrak{X}_C$.
There is a natural map of sites $\nu \colon X_{\mathrm{pro\acute{e}t}} \to \mathfrak{X}_{\mathrm{Zar}}$,
then according to \cite[Definition 8.1 and 9.1]{BMS1}, one defines
\[
A\Omega_{\mathfrak{X}} \coloneqq L\eta_{\mu}(R\nu_* \mathbb{A}_{\inf, X}) \text{ and }
\widetilde{\Omega}_{\mathfrak{X}} \coloneqq L\eta_{\mu}(R\nu_* \mathbb{A}_{\inf, X}/\widetilde{\xi}).
\]
For the purpose of this paper, we merely view the above as objects in
$D(\mathfrak{X}_{\mathrm{Zar}}, A_{\inf})$.
The $A_{\inf}$ cohomology is then defined as
\[
\mathrm{R\Gamma}_{A_{\inf}}(\mathfrak{X}) \coloneqq 
\mathrm{R\Gamma}(\mathfrak{X}_{\mathrm{Zar}}, A\Omega_{\mathfrak{X}}).
\]
By \cite[Theorem 1.8]{BMS1}, all cohomology groups are Breuil--Kisin--Fargues modules
(see \cite[Definition 4.22]{BMS1}). Analogous to \cref{remark on structure of BK module},
using \cite[Proposition 4.13]{BMS1}, we see that
$M^i \coloneqq \mathrm{H}^i(\mathrm{R\Gamma}_{A_{\inf}}(\mathfrak{X}))$ also admits a natural
exact sequence:
\[
0 \to M^i_{\mathrm{tors}} = M^i[p^{\gg 0}] \to M^i \to M^i_{\mathrm{free}} \to 
\overline{M^i} \to 0,
\]
with all modules appearing above, regarded as $A_{\inf}$-complexes, perfect.

In general, (derived) reduction modulo an element certainly does not commute with $L\eta$
with respect to another element. Therefore it is surprising to learn (see \cite[Theorem 9.2.(1)]{BMS1}) 
that the natural map $A\Omega_{\mathfrak{X}}/\widetilde{\xi} \to \widetilde{\Omega}_{\mathfrak{X}}$
is a quasi-isomorphism!
In \cite{Bha18},
at least if we work at the level of almost mathematics with respect
to $[\mathfrak{m}_C^{\flat}]$,
one finds a conceptual proof for this fact.
\end{remark}

\begin{proposition}[{\cite[Lemma 5.16 and Proposition 7.5]{Bha18}}]
\label{Bhatt's proof}
The natural map $A\Omega_{\mathfrak{X}}/\widetilde{\xi} \to \widetilde{\Omega}_{\mathfrak{X}}$
is an almost, with respect to $[\mathfrak{m}_C^{\flat}]$,
isomorphism in $D(\mathfrak{X}_{\mathrm{Zar}}, A_{\inf}^a)$.
\end{proposition}

Let us sketch the proof for later use.

\begin{proof}[Sketch of proof in loc.~cit.]
The Lemma 5.16 in loc.~cit.~provides such a natural map, as well as a criterion
for when the map is an almost isomorphism: it suffices for the cohomology sheaves of
$R\nu_*(\mathbb{A}_{\inf, X})/\mu$ to be almost $\widetilde{\xi}$-torsionfree.
Since $\widetilde{\xi} = \frac{(\mu + 1)^p - 1}{\mu} = \mu^{p-1} + \ldots + p \cdot \mu + p \equiv p$
modulo $\mu$, it is equivalent to these cohomology sheaves being almost $p$-torsionfree.
This later claim follows from Theorem 4.14 and Lemma 7.1 in loc.~cit.
\end{proof}

The above admits a direct generalization.
\begin{proposition}
\label{commuting reduction with Leta}
Define $\widetilde{\Omega}^{(n)}_{\mathfrak{X}} \coloneqq
L\eta_{\mu}(R\nu_* \mathbb{A}_{\inf, X}/\widetilde{\xi}^n) \in D(\mathfrak{X}_{\mathrm{Zar}}, A_{\inf})$.
Then the natural map
$A\Omega_{\mathfrak{X}}/\widetilde{\xi}^n \to \widetilde{\Omega}^{(n)}_{\mathfrak{X}}$
is an almost, with respect to $[\mathfrak{m}_C^{\flat}]$,
isomorphism in $D(\mathfrak{X}_{\mathrm{Zar}}, A_{\inf}^a)$.
\end{proposition}

\begin{proof}
Using again \cite[Lemma 5.16]{Bha18}, we are reduced to showing that
the cohomology sheaves of $R\nu_*(\mathbb{A}_{\inf, X})/\mu$ are almost $\widetilde{\xi}^n$-torsionfree.
Since this is equivalent to these sheaves being almost $\widetilde{\xi}$-torsionfree,
we are done thanks to the proof of \Cref{Bhatt's proof}.
\end{proof}

\begin{lemma}
\label{boundedness of Ainf torsion}
Set $M^i \coloneqq \mathrm{H}^i(\mathrm{R\Gamma}_{A_{\inf}}(\mathfrak{X}))$,
then there exists an $N \gg 0$ such that
$M^i[\widetilde{\xi}^{\infty}] = M^i[\widetilde{\xi}^N]$.
Moreover $M^i[\widetilde{\xi}^{\infty}]$ is a finitely presented
coherent $A_{\inf}$-module.
\end{lemma}

\begin{proof}
By \Cref{Ainf remark}, there exists an $m \in \mathbb{N}$ such that
the torsion submodule in $M^i$ is given by $M \coloneqq M^i[p^m]$,
which is a perfect complex. In particular, it is finitely presented.
Using \cite[Lemma 3.26]{BMS1}, we know that $W_m(\mathcal{O}_C^{\flat})$ is a coherent ring.
By \cite[\href{https://stacks.math.columbia.edu/tag/05CX}{Tag 05CX}]{stacks-project},
we see that $M$ is a coherent $W_m(\mathcal{O}_C^{\flat})$-module.
Therefore we are reduced to showing: if $M$ is a finitely presented
$W_m(\mathcal{O}_C^{\flat})$-module, then there exists an $N \gg 0$
such that $M[\widetilde{\xi}^{\infty}] = M[\widetilde{\xi}^N]$.
Indeed, we may then apply \cite[\href{https://stacks.math.columbia.edu/tag/05CW}{Tag 05CW}]{stacks-project}
to see that $M[\widetilde{\xi}^N] = \ker(M \xrightarrow{\widetilde{\xi}^N} M)$
is a finitely presented coherent $W_m(\mathcal{O}_C^{\flat})$-module.

Let us prove the above claim, by induction on the smallest power $p^m$ of $p$ that annihilates $M$.
If $M$ is annihilated by $p$, this follows from the fact that $\mathcal{O}_C^{\flat}$ is a
rank one valuation ring.
Since $W_m(\mathcal{O}_C^{\flat})$ is a coherent ring, we know that both
$Q \coloneqq M[p]$ and $M/Q \cong \mathrm{Im}(M \xrightarrow{\cdot p} M)$
are finitely presented $W_m(\mathcal{O}_C^{\flat})$-modules.
By induction, if $m > 1$, we see that the $\widetilde{\xi}^{\infty}$-torsion parts in both
$Q$ and $M/Q$ are annihilated by $\widetilde{\xi}^{N'}$
for some $N' \gg 0$.
%In particular, we see that $M/Q[\widetilde{\xi}^{\infty}] = M/Q[\widetilde{\xi}^{N'}]$
%for some $N'$, and this module is also finitely presented over the coherent
%ring $W_m(\mathcal{O}_C^{\flat})$.
By the snake lemma, there is a natural exact sequence
\[
0 \to Q[\widetilde{\xi}^{\infty}] \to M[\widetilde{\xi}^{\infty}] \to M/Q[\widetilde{\xi}^{\infty}].
\]
%Therefore we just need to know that any map $M/Q[\widetilde{\xi}^{\infty}] \to Q[1/\widetilde{\xi}]/Q$
%must have finitely presented image. Since $Q$ is finitely presented $\mathcal{O}_C^{\flat}$-module,
%we see that $Q[1/\widetilde{\xi}]/Q \simeq (C^{\flat}/\mathcal{O}_C^{\flat})^{\oplus r}$.
%One sees directly that any finitely generated submodule in this must be finitely presented,
One immediately sees that $M[\widetilde{\xi}^{\infty}]$ is annihilated
by $\widetilde{\xi}^{2N'}$, hence we are done.
\end{proof}

The following is inspired by the proof of \cite[Lemma 5.1]{Min21}.

\begin{proposition}
\label{Analogue of Min's result}
Let $i > 0$ and set $M^i \coloneqq \mathrm{H}^i(\mathrm{R\Gamma}_{A_{\inf}}(\mathfrak{X}))$,
then $M^i[\widetilde{\xi}^{\infty}]$ is almost, with respect to $[\mathfrak{m}_C^{\flat}]$,
annihilated by $\mu^{i-1}$.
In particular, let $J_{\inf} \subset A_{\inf}$ be the annihilator of $M^i[\widetilde{\xi}^{\infty}]$,
then we have an inclusion $\mu^{i-1} \cdot [\mathfrak{m}_C^{\flat}] \subset J_{\inf}$.
\end{proposition}

\begin{proof}
Let $n$ be an arbitrary positive integer.
Recall \cite[Corollary 6.5]{BMS1} that the $L\eta$ functor commutes with canonical truncation.
Applying \cite[Lemma 6.9]{BMS1}, we see that there is a commutative diagram
in $D(\mathfrak{X}_{\mathrm{Zar}}, A_{\inf}^a)$:
\[
\xymatrix{
\tau^{\leq (i-1)}A\Omega_{\mathfrak{X}} \ar[r]
\ar@/_1pc/[d]_{f_1} & \tau^{\leq (i-1)}\widetilde{\Omega}^{(n)}_{\mathfrak{X}} \ar@/_1pc/[d]_{f_2} \\
\tau^{\leq (i-1)}R\nu_*(\mathbb{A}_{\inf, X}) \ar[r] \ar@/_1pc/[u]_{g_1} & 
\tau^{\leq (i-1)}R\nu_*(\mathbb{A}_{\inf, X}/\widetilde{\xi}^n), \ar@/_1pc/[u]_{g_2}
}
\]
where both horizontal arrows are induced by $\tau^{(i-1)}$ applied to the (derived) reduction
modulo $\widetilde{\xi}^n$ map,
and the composition of $f_j$ and $g_j$ in either direction is $\mu^{i-1}$
for $j = 1, 2$.

By \cite[Theorem 5.1 and proof of Theorem 8.4]{Sch13}, we get almost isomorphisms
\[
\mathrm{R\Gamma}(X_{\mathrm{\acute{e}t}}, \mathbb{Z}_p) \otimes_{\mathbb{Z}_p} A_{\inf} \cong
\mathrm{R\Gamma}(X_{\mathrm{pro\acute{e}t}}, \mathbb{A}_{\inf})
\text{ and } \mathrm{R\Gamma}(X_{\mathrm{\acute{e}t}}, \mathbb{Z}_p) \otimes_{\mathbb{Z}_p} A_{\inf}/\widetilde{\xi}^n 
\cong \mathrm{R\Gamma}(X_{\mathrm{pro\acute{e}t}}, \mathbb{A}_{\inf}/\widetilde{\xi}^n)
\]
with respect to $[\mathfrak{m}_C^{\flat}]$.
Now we take $(i-1)$-st cohomology of the diagram above, and arrive at the following commutative diagram
of almost $A_{\inf}$-modules:
\[
\xymatrix{
\mathrm{H}^{i-1}_{A_{\inf}}(\mathfrak{X}) \ar[r]
\ar@/_1pc/[d]_{f_1} & \mathrm{H}^{i-1}(\mathfrak{X}, \widetilde{\Omega}^{(n)}_{\mathfrak{X}})
\cong \mathrm{H}^{i-1}(\mathrm{R\Gamma}_{A_{\inf}}(\mathfrak{X})/\widetilde{\xi}^n) \ar@/_1pc/[d]_{f_2} \\
\mathrm{H}^{i-1}(X_{\mathrm{\acute{e}t}}, \mathbb{Z}_p) \otimes_{\mathbb{Z}_p} A_{\inf}
\ar@{->>}[r] \ar@/_1pc/[u]_{g_1} & 
\mathrm{H}^{i-1}(X_{\mathrm{\acute{e}t}}, \mathbb{Z}_p) \otimes_{\mathbb{Z}_p} A_{\inf}/\widetilde{\xi}^n, \ar@/_1pc/[u]_{g_2}
}
\]
where the identification of top-right item uses \Cref{commuting reduction with Leta},
and the composition of $f_j$ and $g_j$ in either direction is $\mu^{i-1}$
for $j = 1, 2$.
Since the cokernel of the top arrow is, as an almost $A_{\inf}$-module, given by
$\mathrm{H}^{i}_{A_{\inf}}(\mathfrak{X})[\widetilde{\xi}^n]$, we see that
$\mathrm{H}^{i}_{A_{\inf}}(\mathfrak{X})[\widetilde{\xi}^n]$
is almost annihilated by $\mu^{i-1}$.
By \Cref{boundedness of Ainf torsion}, we can choose $n$ large enough so that 
$\mathrm{H}^{i}_{A_{\inf}}(\mathfrak{X})[\widetilde{\xi}^n] = 
\mathrm{H}^{i}_{A_{\inf}}(\mathfrak{X})[\widetilde{\xi}^\infty]$.
\end{proof}

\begin{lemma}
\label{annihilator ideal of finitely presented module over coherent ring}
Let $R$ be a coherent ring, and let $M$ be a finitely presented $R$-module.
Then the annihilator ideal of $M$ is finitely presented.
\end{lemma}

\begin{proof}
Choose generators $x_i \in M$, each generates a finitely generated submodule
$N_i \coloneqq R \cdot x_i \subset M$. 
By \cite[\href{https://stacks.math.columbia.edu/tag/05CX}{Tag 05CX}]{stacks-project},
the module $M$ is coherent, hence the $N_i$'s are all finitely presented.
Hence we see that each $x_i$ has a finitely generated annihilator ideal $J_i$.
As $R$ is coherent, they are automatically finitely presented.
Finally, it suffices to show that the intersection of two finitely presented ideals in $R$
is again finitely presented.
This follows from applying \cite[\href{https://stacks.math.columbia.edu/tag/05CW}{Tag 05CW}]{stacks-project}
to $J_1 \cap J_2 = \ker(J_1 \to R/J_2)$.
\end{proof}

\begin{corollary}
\label{Corollary of Min's argument}
With setup and notation as in \Cref{Analogue of Min's result}.
The ideal $J_{\inf} \subset A_{\inf}$ is a finitely generated ideal containing some power of $p$,
therefore we in fact have $\mu^{i-1} \in J_{\inf}$.
\end{corollary}

\begin{proof}
By \Cref{boundedness of Ainf torsion} and its proof, we see that
$M^i[\widetilde{\xi}^{\infty}]$ is a finitely presented $W_m(\mathcal{O}_C^{\flat})$-module,
and the ideal $J_{\inf}$ is the preimage under the projection 
$A_{\inf} \xrightarrow{\text{mod } p^m} W_m(\mathcal{O}_C^{\flat})$ of the annihilator ideal
$J' \subset W_m(\mathcal{O}_C^{\flat})$ of $M^i[\widetilde{\xi}^{\infty}]$.
Hence it suffices to know that $J'$ is finitely presented, which follows from
combining \Cref{boundedness of Ainf torsion}
and \Cref{annihilator ideal of finitely presented module over coherent ring}.

It remains to show that $\mu^{i-1} \in J'$, which is equivalent to 
$W_m(\mathcal{O}_C^{\flat}) = \ker(W_m(\mathcal{O}_C^{\flat}) \xrightarrow{\cdot \mu^{i-1}} W_m(\mathcal{O}_C^{\flat})/J')$.
Using \cite[\href{https://stacks.math.columbia.edu/tag/05CW}{Tag 05CW}]{stacks-project}
we see that the kernel is a finitely generated ideal.
By \Cref{Analogue of Min's result}, we see this finitely generated ideal
contains the image of $[\mathfrak{m}_C^{\flat}]$, therefore it must be the unit ideal.
\end{proof}

\section{Applications}
\label{AG section}

Throughout this section, let $\mathcal{X}$ be a smooth proper
formal scheme over $\mathrm{Spf}(\mathcal{O}_K)$.
In this section, we deduce consequences of the previous sections.
We begin with an auxilliary lemma.

\begin{lemma}
\label{Newton polygon lemma}
Let $C$ be a complete algebraically closed nonarchimedean extension
of $\mathbb{Q}_p$.
Let $v_{C^{\flat}}$ be the valuation on the tilt $C^{\flat}$,
normalized so that $v_{C^{\flat}}(p^{\flat}) = 1$.
Let $j > 0$ and consider the Teichm\"{u}ller expansion
\[
\mu^{j} = \sum_{i \geq 0} p^i \cdot [x_i^{(j)}] \in W(\mathcal{O}_C^{\flat}),
\]
then we have $v_{C^{\flat}}(x_{j \ell}^{(j)}) = j \cdot \frac{p}{p^{\ell}(p-1)}$
for any $\ell \in \mathbb{N}$.
\end{lemma}

\begin{proof}
Recall that the addition in ($p$-typical) Witt vectors of a perfect ring $R$
is defined in the following manner.
First there are universal polynomials $Q_i(X, Y) \in \mathbb{Z}[X, Y]$ defined inductively by 
\[
X^{p^n} + Y^{p^n} = \sum_{i = 0}^n p^i Q_i^{p^{n-i}}.
\]
Then we have 
\[
[x] + [y] = \sum_{i \geq 0} p^i \cdot [Q_i(x^{1/p^i}, y^{1/p^i})] \text{ in } W(R)
\]
for any $x, y \in R$.
We can inductively see that
\begin{itemize}
\item Each $Q_i(X, Y)$ is a homogeneous degree $p^i$ polynomial;
\item $Q_0(X, Y) = X + Y$;
\item whenever $i > 0$ there is an expansion of the form $Q_i(X, Y) = \sum_{1 \leq m \leq p^i - 1} a_m X^m Y^{p^i - m}$
with $a_1 = a_{p^i - 1} = 1$.
\end{itemize}

For $x_i \coloneqq x_i^{(1)}$, we have from the above two expansions
\[
[\epsilon] + \sum_{i > 0} p^i [Q_i((\epsilon - 1)^{1/p^i}, 1)] =
[\epsilon - 1] + 1 = [\epsilon] + (-1) \cdot \sum_{i > 0} p^i[x_i],
\]
and $x_0 = \epsilon - 1$.
In particular, we see that
\[
(-1) \cdot \sum_{i > 0} p^i [Q_i((\epsilon - 1)^{1/p^i}, 1)] = \sum_{i > 0} p^i[x_i].
\]
We claim that our lemma follows from this equality, together with the discussions of
``Newton polygon'' in \cite[Subsection 1.5]{FFbook}. 

Let us first summarize necessary definitions and facts concerning Newton polygons:
In \cite[Definition 1.5.2]{FFbook},
to any element $y = \sum_{i \geq 0} p^i \cdot [y_i] \in A_{\inf}$,
the authors define $\mathcal{N}ewt(y)$ to be the function
$\mathbb{R} \to \mathbb{R} \cup \{\infty\}$ whose graph is the highest convex non-increasing polygon
below the points $\{(n, v_{C^{\flat}}(y_n)) \mid n \in \mathbb{N}\}$.
By how $\mathcal{N}ewt(y)$ is defined, we see that if
$(n, v_n)$ is a turning point of its graph, then $v_{C^{\flat}}(y_n) = v_n$.
On \cite[p.~20]{FFbook}, 
the authors conclude that $\mathcal{N}ewt(y \cdot z) = \mathcal{N}ewt(y) * \mathcal{N}ewt(z)$,
where the operation $*$ of convex functions is defined on \cite[p.~18]{FFbook}.
Using this, one checks that
$\mathcal{N}ewt(u \cdot y) = \mathcal{N}ewt(y)$ if $u$ is a unit.

Now we are ready to prove the claim for $j = 1$: Using the previous paragraph,
we see that 
\[
\mathcal{N}ewt(\sum_{i > 0} p^i[x_i]) = 
\mathcal{N}est(\sum_{i > 0} p^i [Q_i((\epsilon - 1)^{1/p^i}, 1)]).
\]
By the third observation on these $Q_i$'s,
we have $v_{C^{\flat}}(Q_i((\epsilon - 1)^{1/p^i}, 1)) = \frac{p}{p^i(p-1)}$
for all $i > 0$.
So the Newton polygon goes precisely through $\{(n, \frac{p}{p^{n-1}(p-1)}) \mid n \in \mathbb{N}\}$
for all $n \geq 1$, and these points are all turning points.
In the end we deduce that 
$v_{C^{\flat}}(x_i) = \frac{p}{p^i(p-1)}$
for all $i > 0$ as well.

The $j=1$ case implies the general case, as follows: Chasing through the definition of $*$,
the graph of $\mathcal{N}ewt(y^j)$ is the original graph of $\mathcal{N}ewt(y)$
scaled by $j$-times. 
Therefore the turning points of $\mathcal{N}ewt(\mu^j)$ are given by
$\{(j \cdot n, j \cdot \frac{p}{p^{n-1}(p-1)}) \mid n \in \mathbb{N}\}$,
finishing the proof.
\end{proof}

With the above preparation, we can prove the following.

\begin{theorem}
\label{consequence of Leta}
Let $i > 0$, denote $\mathfrak{M}^i \coloneqq \mathrm{H}^i_{\Prism}(\mathcal{X}/\mathfrak{S})$,
and let $J$ be the annihilator ideal of $\mathfrak{M}^i_n[u^\infty]$.
Let $\rho$ be defined by $J + (u) = (u, p^{\rho})$.
If $e \cdot (i-1) < p^n(p-1)$, then $\rho \leq (i-1) \cdot n$.
\end{theorem}

By \cite[Corollary 3.8 or Remark 3.9]{LL23}, the $\mathfrak{M}^1$ is always finite free.
Therefore in the following proof, we always assume that $i \geq 2$, hence
$(i-1) > 0$. So we may summon \Cref{Newton polygon lemma} for $j = (i-1)$.

\begin{proof}
Let $\mathfrak{X} \coloneqq \mathcal{X}_{\mathcal{O}_C}$,
and set $M^i \coloneqq \mathrm{H}^i(\mathrm{R\Gamma}_{A_{\inf}}(\mathfrak{X}))$.
After choosing compatible $p$-power roots of $\pi$ in $\mathcal{O}_C$, we get an element
$\pi^{\flat} \in \mathcal{O}_C^{\flat}$ (see \Cref{tilt notation}).
We may consider the map of prisms which is $p$-completely faithfully flat:
\[
f \colon (\mathfrak{S} = W[\![u]\!], (E)) \to (A_{\inf}, (\xi)),
\]
with $f(u) = [\pi^{\flat}]$.
By \cite[Theorem 1.8.(5) and Theorem 17.2]{prism}, we get a canonical isomorphism
$\mathfrak{M}^i \otimes_{\mathfrak{S}, \varphi \circ f} A_{\inf} \cong M^i$.
Using the $p$-completely flatness of $f$, together with structural
results mentioned in \cref{remark on structure of BK module} and \cref{Ainf remark},
we also get
$\mathfrak{M}^i[u^\infty] \otimes_{\mathfrak{S}, \varphi \circ f} A_{\inf} \cong M^i[\widetilde{\xi}^{\infty}]$.
In particular, using again the $p$-completely flatness of $f$,
the annihilator ideal $J_{\inf}$ of $M^i[\widetilde{\xi}^{\infty}]$ 
is given by $(\varphi \circ f)(J) \cdot A_{\inf}$.

Now suppose that $\rho > (i-1) \cdot n$, then we have $J \subset (u, p^{(i-1) \cdot n + 1})$,
consequently $J_{\inf} \subset J'_{\inf} \coloneqq ([(\pi^{\flat})^p]), p^{(i-1) \cdot n + 1})$.
Notice that an element $x = \sum_{m \geq 0} p^m \cdot [x_m] \in A_{\inf}$ lies in $J'_{\inf}$
if and only if $v_{C^{\flat}}(x_m) \geq \frac{p}{e}$ for all $m \leq (i-1) \cdot n$.
%From this description, one verifies directly that if an element $x$ is almost in $J'_{\inf}$
%(i.e.~$x \cdot [\mathfrak{m}^{\flat}_C] \subset J'_{\inf}$)
%if and only if $x \in J'_{\inf}$.

\Cref{Corollary of Min's argument} says that $\mu^{i-1} \in J_{\inf}$.
Therefore by the above paragraph, in the Teichm\"{u}ller expansion of
$\mu^{i-1} = \sum_{m \geq 0} p^m \cdot [x_m^{(i-1)}]$, we must have
$v_{C^{\flat}}(x_m^{(i-1)}) \geq \frac{p}{e}$ for all $m \leq (i-1) \cdot n$.
On the other hand, by \Cref{Newton polygon lemma} we have 
\[
v_{C^{\flat}}(x_{(i-1) \cdot n}^{(i-1)}) = (i-1) \cdot \frac{p}{p^n(p-1)}
\]
contradicting with the assumption $e \cdot (i-1) < p^n(p-1)$.
\end{proof}

In practice, it is also important to understand the cohomology of 
$\mathrm{R\Gamma}_{\Prism}(\mathcal{X}/\mathfrak{S})/^L p^n)$.
The above proof no longer works in this generality, but we have
arguments purely from commutative algebra, at the expense
of getting slightly worse bound.

\begin{theorem}
\label{annihilator of u-power torsion}
Let $n \in \mathbb{N} \cup \{\infty\}$ and let $i > 0$,
denote $\mathfrak{M}^i_n \coloneqq \mathrm{H}^i(\mathrm{R\Gamma}_{\Prism}(\mathcal{X}/\mathfrak{S})/^L p^n))$ (where $n = \infty$ means that we do not perform the reduction at all),
and let $J$ be the annihilator ideal of $\mathfrak{M}^i_n[u^\infty]$.
Lastly, let $\sigma$ and $\rho$ be defined by $J + (p) = (p, u^{\sigma})$ and
$J + (u) = (u, p^{\rho})$, we have
\begin{enumerate}
\item inequalities $\sigma \leq \lfloor\frac{e \cdot (i-1)}{p-1}\rfloor$ and $\rho \leq d(e, i-1)$;
\item a belonging $p^{(\rho + i) \cdot \sigma} \in J$; and
\item an inclusion $(u, p)^{(\rho + i) \cdot \sigma^2} \subset J$.
\end{enumerate}
\end{theorem}

\begin{proof}
Using \Cref{Consequence of Nygaard filtrations},
the statement (1) follows from \Cref{bound on rho},
the statement (2) follows from \Cref{boundedness proposition},
whereas the statement (3) follows from the combination of (1) and (2).
\end{proof}

In \cite{LL23} one finds results relating pathologies in $p$-adic
geometry with $u$-torsion in prismatic cohomology, here let us
update the conclusions with our new estimates.

\begin{proposition}
Assume that the formal scheme $\mathcal{X}$ has an $\mathcal{O}_K$-point.
Let $f \colon \mathrm{Alb}(\mathcal{X}_k) \to \mathrm{Alb}(\mathcal{X}_K)_k$ be the natural map discussed in the beginning
of \cite[Subsection 4.1]{LL23}.
Then $\ker(f)$ is $p^n$-torsion if $e < p^n(p-1)$.
\end{proposition}

\begin{proof}
This follows from combination of \Cref{consequence of Leta}, \cite[Proposition 4.1]{LL23}
and \cite[Theorem 4.2]{LL23}.
\end{proof}

%\begin{remark}
%The examples and discussions made in \cite[Section 6]{LL23}
%shows that our bound is sharp when $p \not= 2$.
%It seems that our estimate of $\rho$ should be made $1$ better when $p = 2$,
%but our methods in the previous section fail to achieve this.
%\textcolor{red}{This remark is now redundant!}
%\end{remark}

\begin{proposition}
Let $C$ be the completion of an algebraic closure of $K$,
let $n \in \mathbb{N} \cup \{\infty\}$ and let $i > 0$,
consider the specialization map
\[
\mathrm{Sp}^i_n \colon \mathrm{H}^i_{\acute{e}t}(\mathcal{X}_{\bar{k}}, \mathbb{Z}/p^n) \to \mathrm{H}^i_{\acute{e}t}(\mathcal{X}_C, \mathbb{Z}/p^n)
\]
discussed in the beginning
of \cite[Subsection 4.2]{LL23}
(here again $n = \infty$ means that we do not perform reduction at all).
Then $\ker(\mathrm{Sp}^i_n)$ is $p^{d(e, i-1)}$-torsion.
\end{proposition}

\begin{proof}
This follows from \Cref{annihilator of u-power torsion} and \cite[Theorem 4.14]{LL23}.
\end{proof}

From now on, we use the notation from \Cref{remark on structure of BK module}.
Let us observe that one can control $\overline{\mathfrak{M}^i}$ in terms of $\mathfrak{M}^i/p^N$ for some $N \gg 0$.

\begin{lemma}
\label{barM is a natural sub}
Let $\mathfrak{M}$ be any finitely generated $\mathfrak{S}$-module admitting exact sequences as 
in \Cref{remark on structure of BK module},
let $p^m$ be such that it annihilates both $\mathfrak{M}_{\mathrm{tors}}$ and $\overline{\mathfrak{M}}$,
then there is an exact sequence:
\[
0 \to \mathfrak{M}_{\mathrm{tors}} \oplus \overline{\mathfrak{M}} \to \mathfrak{M}/p^{N} \to 
(\mathfrak{M})^{\vee \vee}/p^N \to \overline{\mathfrak{M}} \to 0,
\]
for any $N \geq 2m$. In particular, there is an identification
$\mathfrak{M}/p^{N}[u^\infty] \simeq \mathfrak{M}[u^\infty] \oplus \overline{\mathfrak{M}}$ whenever $N \geq 2m$.
\end{lemma}

\begin{proof}
For any natural number $n$, we have canonical exact sequences:
\[
0 \to \mathfrak{M}_{\mathrm{tors}}/p^n \to \mathfrak{M}/p^n \to \mathfrak{M}_{\mathrm{tf}}/p^n \to 0,
\]
\[
0 \to \overline{\mathfrak{M}}[p^n] \to \mathfrak{M}_{\mathrm{tf}}/p^n \to (\mathfrak{M})^{\vee \vee}/p^n \to \overline{\mathfrak{M}}/p^n \to 0.
\]
The second sequence implies that $\mathfrak{M}_{\mathrm{tf}}/p^n[u^\infty] \cong \overline{\mathfrak{M}}[p^n]$,
with the natural transitions map from $(n+1)$-st level to $n$-th level on the left hand side identified with
the multiplication by $p$ map on the right hand side.
Now let us denote $\mathfrak{N}_n \coloneqq \{x \in \mathfrak{M}/p^n \mid u^{\gg 0} x \in \mathfrak{M}_{\mathrm{tors}}/p^n \subset \mathfrak{M}/p^n\}$,
then we have a canonical isomorphism $\left(\mathfrak{M}/p^n\right)/\mathfrak{N}_n \cong
\left(\mathfrak{M}_{\mathrm{tf}}/p^n\right)/\overline{\mathfrak{M}}[p^n]$ and 
commutative diagrams of exact sequences:
\[
\xymatrix{
0 \ar[r] & \mathfrak{M}_{\mathrm{tors}}/p^{n+1} \ar[d]_-{\text{mod }p^n} \ar[r] & \mathfrak{N}_{n+1} \ar[d] \ar[r]
& \overline{\mathfrak{M}}[p^{n+1}] \ar[d]^{\cdot p} \ar[r] & 0 \\
0 \ar[r] & \mathfrak{M}_{\mathrm{tors}}/p^{n} \ar[r] & \mathfrak{N}_{n} \ar[r] & \overline{\mathfrak{M}}[p^{n}] 
\ar[r] & 0.
}
\]
If we consider the transition map from the $N$-th level to the $m$-th level, we get a splitting
$\mathfrak{N}_N \simeq \mathfrak{M}_{\mathrm{tors}} \oplus \overline{\mathfrak{M}}$. The two sequences in the beginning
combine into
\[
0 \to \mathfrak{N}_n \to \mathfrak{M}/p^n \to (\mathfrak{M})^{\vee \vee}/p^n \to \overline{\mathfrak{M}}/p^n \to 0.
\]
This finishes the proof.
\end{proof}

\begin{corollary}
\label{Bounding barM}
Let $J'$ be the annihilator of $\overline{\mathfrak{M}^i}$ with $i > 0$, and let $\rho'$ be such that
$J' + (u) = (u, p^{\rho'})$.
Then we have $\rho' \leq d(e, i-1)$.
\end{corollary}

\begin{proof}
Since we have a natural injection $\mathfrak{M}^i/p^n \hookrightarrow \mathfrak{M}^i_n$,
this follows from \Cref{barM is a natural sub} and \Cref{annihilator of u-power torsion}.
\end{proof}

Lastly we present our ultimate application:
\begin{theorem}
\label{bounding crystalline torsion}
There exists a constant $c(e, i)$ depending only on ramification index $e$ and cohomological degree $i > 0$,
such that if the $\mathrm{H}^i_{\acute{e}t}(\mathcal{X}_C, \mathbb{Z}_p)_{\mathrm{tors}}$ is annihilated by
$p^m$, then the $\mathrm{H}^i_{\mathrm{crys}}(\mathcal{X}_k/W)_{\mathrm{tors}}$ is annihilated by
$p^{m+c}$.
\end{theorem}

\begin{proof}
By \cite[Theorem 1.8.(1)\&(5)]{prism}, we have a natural exact sequence:
\[
0 \to \mathfrak{M}^i/u \to \mathrm{H}^i_{\mathrm{crys}}(\mathcal{X}_k/W) \otimes_{W, \varphi^{-1}} W \to 
\mathfrak{M}^{i+1}[u] \to 0.
\]
By \Cref{annihilator of u-power torsion}, we see that the third term is annihilated by $p^{d(e, i)}$.
We claim that $\left(\mathfrak{M}^i/u\right)_{\mathrm{tors}}$ is annihilated by $p^{m + 2 \cdot d(e, i-1)}$.
Our theorem follows from this claim, by taking $c = 2 \cdot d(e, i-1) + d(e, i)$.

To see our claim:
By \Cref{remark on structure of BK module}, there is a natural exact sequence:
\[
0 \to \mathfrak{M}^i_{\mathrm{tors}}/u \to \left(\mathfrak{M}^i/u\right)_{\mathrm{tors}} \to 
\left(\mathfrak{M}^i_{\mathrm{tf}}/u\right)_{\mathrm{tors}} \cong \overline{\mathfrak{M}^i}[u] \to 0.
\]
By \Cref{Bounding barM}, we see that the third term above is annihilated by $p^{d(e, i-1)}$.
We have reduced our claim to: the $\mathfrak{M}^i_{\mathrm{tors}}/u$ is annihilated by
$p^{m + d(e, i-1)}$.

We have an exact sequence:
\[
0 \to \mathfrak{M}^i[u^\infty] \to \mathfrak{M}^i_{\mathrm{tors}} \to
\mathfrak{M}^i_{\mathrm{tors}, u\mathrm{-tf}} \to 0,
\]
hence the following exact sequence:
\[
0 \to \mathfrak{M}^i[u^\infty]/u \to \mathfrak{M}^i_{\mathrm{tors}}/u \to
\mathfrak{M}^i_{\mathrm{tors}, u\mathrm{-tf}}/u \to 0.
\]
By \Cref{annihilator of u-power torsion} (with $n = \infty$), we see that the first term above
is annihilated by $p^{d(e, i-1)}$.
Lastly by combining \cite[Theorem 1.8.(5) and Section 17]{prism} and \cite[Theorem 1.8.(iv)]{BMS1},
we see that there is a (non-canonical) isomorphism
$\mathfrak{M}^i_{\mathrm{tors}, u\mathrm{-tf}}[1/u] \simeq 
\mathrm{H}^i_{\acute{e}t}(\mathcal{X}_C, \mathbb{Z}_p)_{\mathrm{tors}} \otimes_{\mathbb{Z}_p} \mathfrak{S}[1/u]$,
therefore $\mathfrak{M}^i_{\mathrm{tors}, u\mathrm{-tf}}$ is annihilated by $p^m$, finishing our proof.
\end{proof}

In the above, one could improve the constant $c$ slightly by replacing the bound obtained in
\Cref{annihilator of u-power torsion} with \Cref{consequence of Leta}
at several places. However we feel that the constant $c$ obtained via this method
is unlikely to be optimal anyway, so we do not choose to optimize the bound in the proof
to prevent complicating notations.

\subsection*{Acknowledgements}
This research is supported by the National Key R $\&$ D Program of China No.~2023YFA1009701, and the National Natural Science Foundation of China (No.~12288201).
The second named author benefited from conversations with Bhargav Bhatt, Tong Liu, Yu Min, Sasha Petrov,
Alex Youcis regarding the topic discussed in this article,
he thanks these mathematicians heartily.
%\textcolor{red}{The second author benefited from correspondences
%with the first author, he also wishes to thank the following 
%mathematicians:}

% * End of document
% ** Bibliography
\bibliographystyle{amsalpha}
\bibliography{main}

\end{document}